\documentclass{amsart}
\usepackage{amssymb}

\usepackage{graphicx}

\usepackage{url}

\newtheorem{theorem}{Theorem}[section]
\newtheorem{lemma}[theorem]{Lemma}
\newtheorem{corollary}[theorem]{Corollary}

\newtheorem{claim}[theorem]{Claim}

\theoremstyle{definition}

\theoremstyle{remark}
\newtheorem{remark}[theorem]{Remark}

\numberwithin{equation}{section}

\begin{document}

\title[L--space knots with arbitrary braid index]{Asymmetric L--space knots with arbitrary braid index}



\author{Keisuke Himeno}
\address{Faculty of Education, Yamaguchi University, Yoshida, Yamaguchi, 753-8511, Japan.}
\email{himekei@yamaguchi-u.ac.jp}
\thanks{The first author has been supported by 
JSPS KAKENHI Grant Number JP26K24490.}

\author{Masakazu Teragaito}
\address{Department of Mathematics Education, Hiroshima University,
1-1-1 Kagamiyama, Higashi-hiroshima, 739-8524, Japan.}
\email{teragai@hiroshima-u.ac.jp}
\thanks{The  second author  has been supported by
JSPS KAKENHI Grant Number JP25K07004. }


\date{}


\commby{}

\begin{abstract}
The first examples of asymmetric L--space knots were found by Baker and Luecke.
Among Baker--Luecke knots, the simplest one has braid index $12$.
Later, it turned out that there are just  $9$ asymmetric hyperbolic L--space knots
in the SnapPy census, and their braid indices take the values $4$, $5$, $6$ and $7$.
It is known that asymmetric L--space knots have braid index at least $4$.
Recently, Baker and the second author give infinitely many asymmetric hyperbolic L--space knots
with braid index $4$.

In this paper, 
we construct an asymmetric hyperbolic L--space knot with arbitrary braid index bigger than $4$.
In fact, there exist infinitely many such knots for each braid index.
\end{abstract}

\maketitle

\section{Introduction}\label{sec:intro}

A knot  in the $3$-sphere $S^3$ is called an \textit{L--space knot\/} if it admits
a positive Dehn surgery yielding an L--space.
Typical examples are (positive) torus knots, and there are many known hyperbolic or satellite
L--space knots.
Initially, it was thought that any L--space knot admits symmetry, but
Baker and Luecke \cite{BL} gave the first examples of asymmetric L--space knots. 
Later, it was confirmed that there are only $9$ asymmetric hyperbolic L--space knots in the SnapPy census
\cite{ABG}.
Here, a knot is said to be \textit{asymmetric\/} if the symmetry group of the complement is trivial.

Except Baker--Luecke knots, the mechanism by which an asymmetric L--space knot is formed
has not yet been clarified.
In this paper, we investigate asymmetric L--space knots from the view point of braid index.
It is known that such knots have braid index at least $4$ (see \cite{ABG}).
Among the asymmetric census knots, \texttt{t12533} is the only knot with braid index $4$, and
this is generalized to an infinite family by \cite{BT}.
Also, only the values $4,5,6,7$ are realized by the asymmetric census knots,
and the simplest Baker--Luecke knot has braid index $12$.

The purpose of this paper is to show that any braid index bigger than $4$ is realized
by infinitely many asymmetric hyperbolic L--space knots.
For integers $m\ge 1$ and $n\ge 0$, let
the knot $K_{m,n}$ be the closure of a positive $(m+2)$-braid
\begin{equation}\label{eq:braid}
[ (1,2,\dots, m+1)^{(m+2)n+2}, (m+1,m,\dots,1), 1,2,1,1],
\end{equation}
where an integer $i$ means the standard braid generator $\sigma_i$ in the braid group of $m+2$ strands,
and a power indicates a repetition.
Note that $(1,2,\dots, m+1)^{m+2}$ corresponds to a positive one full twist on $m+2$ strands,
so the braid (\ref{eq:braid}) contains positive $n$ full twists.

For simplicity, we write $K_{m,0}$ as $K_m$.
Then $K_1$ is the torus knot $T(3,5)$, and $K_2$ is the $(-2,3,7)$-pretzel knot.
It turns out that $K_m$ is a strongly invertible L--space knot  for any $m\ge 1$ (see Section \ref{sec:Km}).
Also, $K_{1,1}$ is the torus knot $T(3,8)$, and $K_{2,1}$ is 
\texttt{m239} in the census, which is a strongly invertible L--space knot.
However, $K_{3,1}$ is \texttt{o9\_42675}, which is one of the $9$ asymmetric census L--space knots \cite{ABG}.

Note that $K_{m,n}$ is twist positive if $n\ge 1$  in the words of \cite{KM}.
That is, it is the closure of a positive $(m+2)$-braid with at least one positive full twist.
Then its braid index is $m+2$ by \cite{FW,Mor}.

\begin{theorem}\label{thm:main}
If $m\ge 3$ and $n\ge 1$, then
$K_{m,n}$ is an asymmetric, hyperbolic L--space knot whose braid index is $m+2$.
\end{theorem}

\begin{corollary}
For any integer $b\ge 5$, there exist infinitely many asymmetric hyperbolic L--space knots with braid index $b$ and bridge index $b$.
\end{corollary}

\begin{proof}
For a given $b\ge 5$, set $m=b-2$.
By Theorem \ref{thm:main}, $K_{m,n}$ is an asymmetric hyperbolic L--space knot with braid index $b$.
Since $K_{m,n}$ is twist positive, the braid index and bridge index agree \cite[Theorem 1.3]{KM}
(see also \cite{H}).

Since $K_{m,n}$ is the closure of a positive braid, it is fibered \cite{S}, and its genus is calculated from
the braid presentation (\ref{eq:braid}).
Then $K_{m,n}$ has genus $(m^2+3m+2)n/2+m+3$, so $K_{m,n}\ne K_{m,n'}$ if $n\ne n'$.
\end{proof}

The knots in \cite{BT} are twist positive, so they realize bridge index $4$.
It seems to be unknown whether there exists an asymmetric L--space knot with bridge index $3$.
In fact, Krishna and Morton \cite[Conjecture 1.7]{KM} conjecture that for an L--space knot, the bridge index and braid index agree,
even if the knot is not assumed to be twist positve.
Hence it is plausible that there would not exist an asymmetric L--space knot with bridge index $3$.

Here is the organization of the paper.
In Section \ref{sec:Km}, we prove that $K_m$ is an L--space knot for $m\ge 1$.
Although this case is not included in Theorem \ref{thm:main},
we need it  to establish the fact that $K_{m,n}$ is an L--space knot for $n\ge 1$ in Section \ref{sec:L}.
Finally, we show that $K_{m,n}$ is asymmetric and hyperbolic when $m\ge 3$ and $n\ge 1$ in Section \ref{sec:HA}
by using the result of Futer--Purcell-Schleimer \cite{FPS} and calculations by SnapPy \cite{CD}.

\section{The case where $n=0$.} \label{sec:Km}

In this section, we prove that $K_m$ is an L--space knot for $m\ge 1$.
As usual, $r$-surgery on a knot $k$ is denoted by $k(r)$, and
$(k_1\cup k_2)(r_1,r_2)$ denotes $(r_1,r_2)$-surgery on a link $k_1\cup k_2$.

\begin{theorem}\label{thm:Km}
For $m\ge 1$, $K_m(4m+11)$ is an L--space.
Hence, $K_m$ is an L--space knot.
\end{theorem}

\begin{proof}
We use the link $K\cup c$ as shown in Figure \ref{fig:K+c}.
After performing $(-1)$-surgery on $c$, $K$ is changed into $K_m$.
The link is deformed into a strongly invertible position as shown there.
The box labeled with $2m+1$ contains $2m+1$ vertical right handed half twists.

\begin{figure}[ht]
\includegraphics*[width=0.8\textwidth]{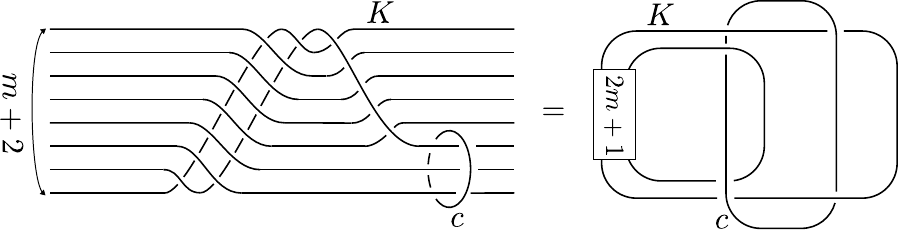}
\caption{Left: The link $K\cup c$, where $K$ is the closure of the $(m+2)$-braid.
By performing $(-1)$-surgery on $c$,
$K$ is changed into the knot $K_m$.
(We draw the case where $m=6$.)
Right: It is deformed into a stronly invertible position.
}
\label{fig:K+c}
\end{figure}

Consider the surgery diagram $(K\cup c)(4m+2,-1)$, which represents $K_m(4m+11)$.
Since the link is strongly invertible, we can apply the Montesinos trick \cite{Mo}.
After taking the quotient under the involution around the axis depicted
by the dotted line in Figure \ref{fig:K+c-mont}, the tangle replacements are performed.
The box labeled with $m$ contains $m$ horizontal right handed half twists, and
the box with $m-1$ contains $m-1$ vertical right handed half twists.
As shown there, the resulting knot is the Montesinos knot $M(-\frac{1}{m-1}, \frac{2}{5},-\frac{2}{3})$ if $m\ge 3$,
the 2-bridge knot $7_6$ if $m=2$,
or the connected sum of the trefoil and the figure eight knot if $m=1$.

\begin{figure}[ht]
\includegraphics*[width=0.8\textwidth]{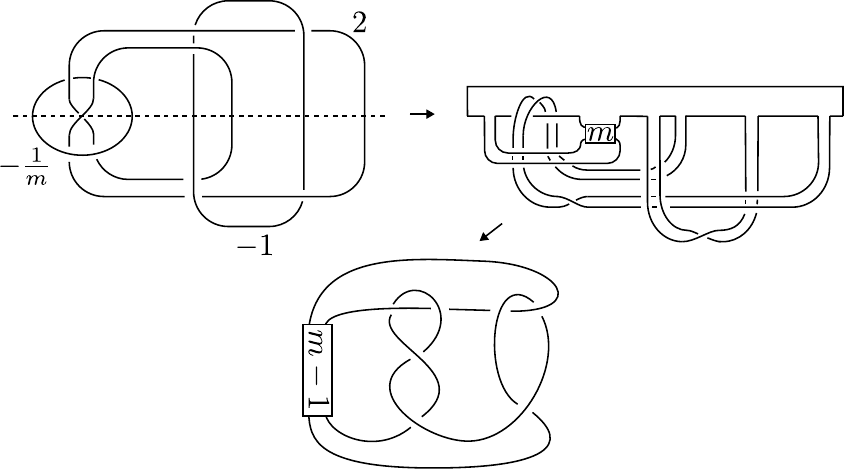}
\caption{
After taking the quotient under the involution around the axis,
perform the tangle replacements.  The result is the Montesinos knot $M(-\frac{1}{m-1}, \frac{2}{5},-\frac{2}{3})$ if $m\ge 3$,
the $2$-bridge knot $7_6$ if $m=2$, or the connected sum of the trefoil and the figure eight knot if $m=1$.
}
\label{fig:K+c-mont}
\end{figure}

Recall that $K_m(4m+11)$ is the double branched cover of this resulting knot.
If $m=1$, then it is the connected sum of two lens spaces, so we are done, because a lens space is an L--space and so is their connected sum \cite{OS2}.
If $m=2$, then it is a lens space.
Assume $m\ge 3$.
Then
 $K_m(4m+11)$ is the double branched cover of this Montesinos link, which 
is the Seifert fibered manifold $M(0;-\frac{1}{m-1}, \frac{2}{5},-\frac{2}{3})$.
Here we use the notation in \cite{LS} (see also \cite{BK}).
The Seifert fibered manifold $M(e_0; r_1,r_2,r_3)$ is obtained by $e_0$-surgery on the unknot with
$3$ meridians having the coefficient $-1/r_i$ on the $i$th one.

We can conclude that our Seifert fibered manifold $M(0;-\frac{1}{m-1}, \frac{2}{5},-\frac{2}{3})$ is an L--space by using
\cite[Theorem 1]{LS}.
It claims that $M=M(e_0; r_1,r_2,r_3)$ (with $1\ge r_1\ge r_2\ge r_3>0$)  is an L--space if and only if either  $M$ or $-M$
carries no positive transeverse contact structure.
Then \cite[Theorem 1.3]{LM} shows that it is equivalent to the condition that $e_0\ge 0$, or $e_0=-1$ and
there are no coprime integers $\ell>a>0$ such that $\ell r_1<a<\ell(1-r_2)$ and $\ell r_3<1$.

Since 
$-M(0;-\frac{1}{m-1}, \frac{2}{5},-\frac{2}{3})$ is homeomorphic to $M(-1; \frac{1}{m-1}, \frac{3}{5}, \frac{2}{3})$,
we can set $e_0=-1$, $r_1=2/3$, $r_2=3/5$ and $r_3=1/(m-1)$.
Obviously, the inequality $\ell r_1<\ell (1-r_2)$ is impossible.
\end{proof}

\begin{remark}
Since the link $K\cup c$ is strongly invertible, we can conclude that $K_m$ is strongly invertible for any $m\ge 1$.
In fact, it is easy to make a strongly invertible diagram of $K_m$ by twisting along $c$.
\end{remark}

\section{The case where $n\ge 1$.} \label{sec:L}

Based on Theorem \ref{thm:Km}, we prove that $K_{m,n}$ is also an L--space knot, in general.
We use the link $K_m\cup A$ as shown in Figure \ref{fig:K+A}.
After $(-1/n)$-surgery on $A$, $K_m$ is changed into $K_{m,n}$.

\begin{figure}[ht]
\includegraphics*[width=0.5\textwidth]{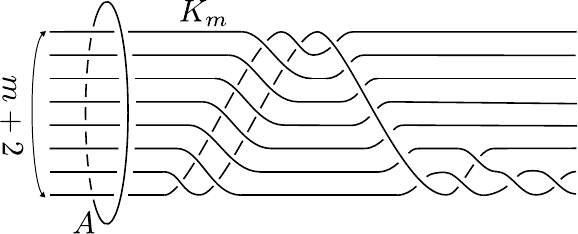}
\caption{
The link $K_m\cup A$.
After $(-1/n)$-surgery on $A$, $K_m$ is changed into $K_{m,n}$.
}
\label{fig:K+A}
\end{figure}

Let us briefly recall the definition of quasi-alternating link \cite{OS1}.
The set $\mathcal{Q}$ of quasi-alternating links is the smallest set of links which satisfies the following.
\begin{itemize}
\item[(1)]
The unknot belongs to $\mathcal{Q}$.
\item[(2)]
If a link $L$ has a crossing $c$ in its diagram $D$ such that both resolutions $D_0$ and $D_\infty$ at $c$
give the links $L_0$ and $L_\infty$ in $\mathcal{Q}$ with $\det L=\det L_0+\det L_\infty$ and $\det L_0, \det L_\infty \ne 0$, then $L$ belongs to $\mathcal{Q}$.
\end{itemize}
It is known that alternating knots and non-split alternating links are quasi-alternating \cite[Lemma 3.2]{OS1}, and the double branched cover of $S^3$ branched over a quasi-alternating link
is an L--space \cite[Proposition 3.3]{OS1}.

\begin{lemma}\label{lem:KA}
$(K_m\cup A)(2m+15,0)$ is an L--space.
\end{lemma}

\begin{proof}
We use the Kirby calculus (\cite{GS}) and the Montesinos trick.
Perform a handle slide as shown in Figure \ref{fig:slide}.
Then the coefficient on $K_m$ is changed to $2m+15-2(m+2)=11$,
since $K_m$ has linking number $m+2$ with $A$ under suitable orientations.

\begin{figure}[ht]
\includegraphics*[width=0.9\textwidth]{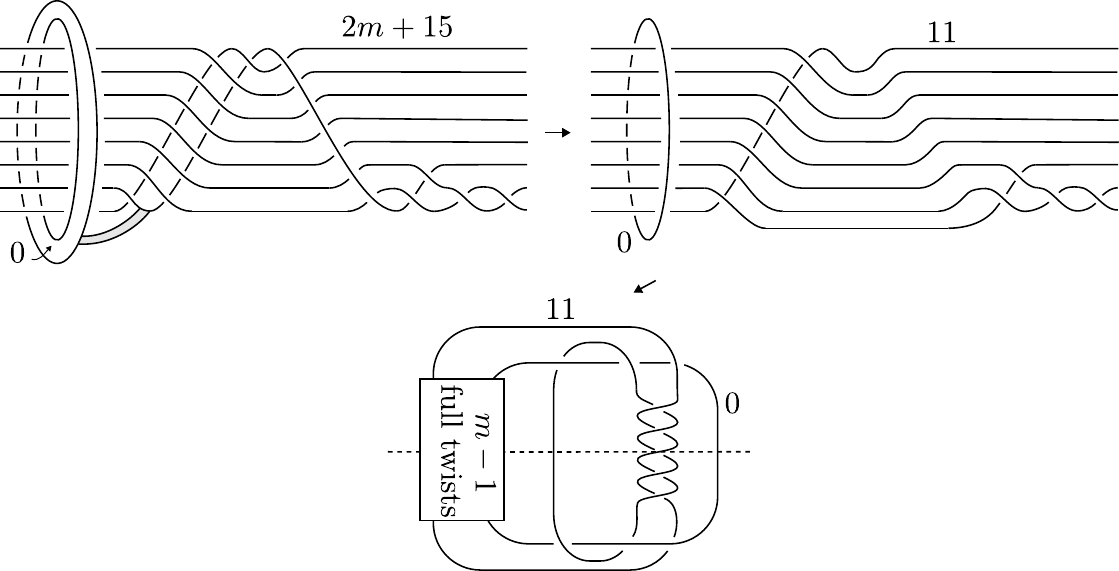}
\caption{
Perform a handle slide.
The result is deformed into a strongly invertible diagram.
}
\label{fig:slide}
\end{figure}

The result is deformed into a strongly invertible diagram, so we can apply the Montesinos trick.
Add two unknotted components to absorb full twists, and take the quotient under the involution
around the axis depicted by the dotted line in Figure \ref{fig:slide2}.
After the tangle replacements, we have a link diagram (Middle bottom in Figure \ref{fig:slide2}),
which is denoted by $D(1-m)$ to record the number of half twists in the rightmost box.
(Here, each box labeled with integer $i$ contains $i$  right handed half twists if $i>0$, or 
$-i$ left handed half twists if $i<0$.)

\begin{figure}[ht]
\includegraphics*[width=0.9\textwidth]{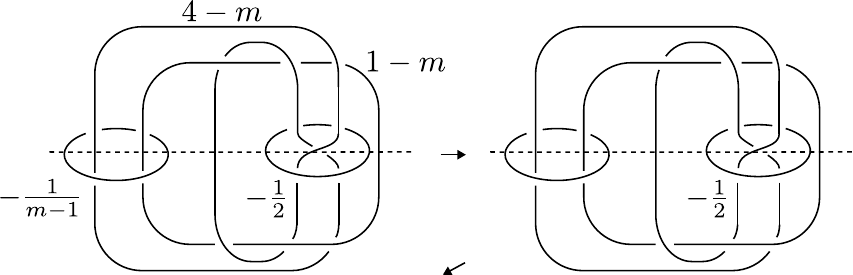}
\caption{
Add two unknotted components to absorb full twists, and take the quotient under the involution
around the axis depicted by the dotted line.
The result is denoted by $D(1-m)$ to record the number of twists in the rightmost box.
}
\label{fig:slide2}
\end{figure}

\begin{claim}\label{cl:QA}
The link $D(1-m)$ is quasi-alternating.
\end{claim}

\begin{proof}[Proof of Claim \ref{cl:QA}]
First, it is straightforward to verify that $D(0)$ is the knot $8_{20}$ and $D(-1)$ is \texttt{L7a3}, which 
is the pretzel link $P(2,2,3)$.  Since these are known to be quasi-alternating, we are done.

Suppose $m\ge 3$.
Note that $\det D(1-m)=(m+2)^2$.
(This can be verified by using a checkerboard coloring of the diagram and its Goeritz matrix.)
Choose any  crossing $c$ in the rightmost box.
By performing two types of resolution at $c$, we obtain the $2$-bridge knot $6_2$ with determinant $11$ and $D(2-m)$
wtih determinant $m^2+4m-7$.
Then we have the identity $\det D(1-m)=\det D(2-m)+\det 6_2$.
Repeating this process, we stop at $D(0)$ with determinant $m^2-7m+15$.

When $m=3$, $D(0)$ is the left handed trefoil.
Also, if $m=4$, $D(0)$ is the right handed trefoil.
Hence we may assume $m\ge 5$.
It is easy to see that $D(0)$ is the pretzel knot $P(m-4,3-m,3)$, and it is quasi-alternating
by \cite[Theorem 3.2]{CK}.
\end{proof}

Since $(K_m\cup A)(2m+15,0)$ is the double branched cover of a quasi-alternating link $D(1-m)$,
it is an L--space.
\end{proof}

\begin{theorem}\label{thm:Kmn}
For $m\ge 1$ and $n\ge 1$,  $K_{m,n}$ is an L--space knot.
\end{theorem}

\begin{proof}
By Theorem \ref{thm:Km}, we know that $K_m$ is an L--space knot.
Since $K_m$ has genus $g(K_m)=m+3$, we have an inequality $2m+15 \ge 2g(K_m)-1$.
According to \cite[Lemma 2.1]{BT} (originally from \cite[Theorem 1.13 and Lemma 6.1]{BM}),
Lemma \ref{lem:KA} under these facts implies that $K_{m,n}$ is an L--space knot for any $n\ge 1$.
\end{proof}


\section{Hyperbolicity and asymmetry}\label{sec:HA}

Finally, we prove that $K_{m,n}$ is hyperbolic and asymmetric for $m\ge 3$ and $n\ge 1$
by using SnapPy.
We use the link $L=K\cup C_0\cup C_1 \cup C_2\cup C_3$ as shown in Figure~\ref{figure_asy_link_diagram}.

\begin{figure}[ht]
\centering
\includegraphics[width=0.5\textwidth]{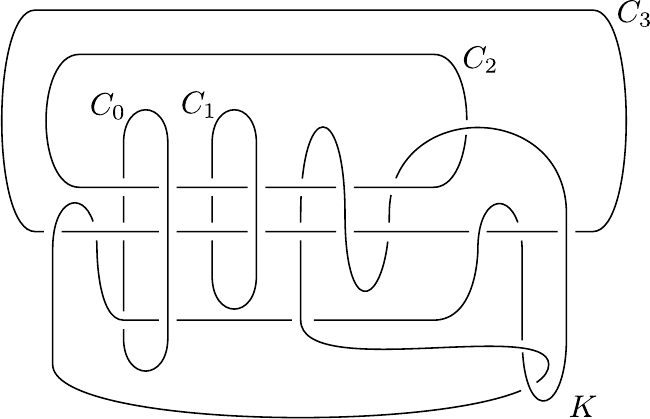}
\caption{The link $L=K\cup C_0\cup C_1 \cup C_2\cup C_3$. By performing $\frac{1}{m}$-surgery on $C_0$, $(-\frac{1}{m})$-surgery on $C_1$, $(-\frac{1}{n-1})$-surgery on $C_2$ and $(-1)$-surgery on $C_3$, $K$ is changed into the knot $K_{m,n}$}
\label{figure_asy_link_diagram}
\end{figure}

\begin{lemma}
Let $m\ge 3$ and $n\ge 1$. By performing $\frac{1}{m}$-surgery on $C_0$, $(-\frac{1}{m})$-surgery on $C_1$, $(-\frac{1}{n-1})$-surgery on $C_2$ and $(-1)$-surgery on $C_3$, $K$ is changed into the knot $K_{m,n}$. 
\end{lemma}

\begin{proof}
Perform $-m$ twists on $C_0$ and $m$ twists on $C_1$ to erase these components.
The coefficients of $C_2$ and $C_3$ remain unchanged, because the changes in the coefficients caused by $(-m)$ twists on $C_0$ are exactly canceled by those caused by $m$ twists on $C_1$.
As a result, we obtain the link shown on the left of Figure~\ref{figure_asy_deform_1}, which can be deformed into the diagram on the right.

Moreover, performe a $1$ twist on $C_3$ to erase this component.
Then we have the link shown on the left of Figure~\ref{figure_asy_deform_2}.
The resulting link $K\cup C_2$ can be deformed into the union of the closure of an $(m + 2)$-braid and its braid axis. It is easy to see that this braid is equivalent to the one given in (\ref{eq:braid}).
Finally, performing $n-1$ twists on $C_2$ yields the knot $K_{m,n}$. 
\end{proof}

\begin{figure}[ht]
\centering
\includegraphics[scale=0.8]{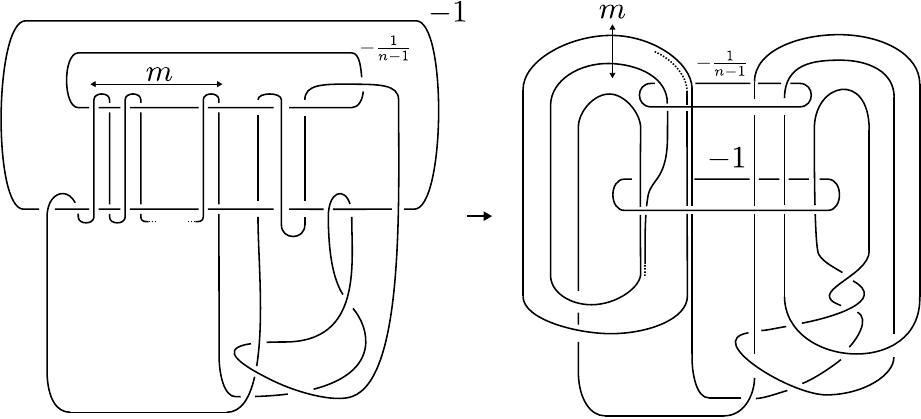}
\caption{The link after performing the twists on $C_0$ and $C_1$ (left), and the deformed link (right).}
\label{figure_asy_deform_1}
\end{figure}

\begin{figure}[ht]
\centering
\includegraphics[scale=0.8]{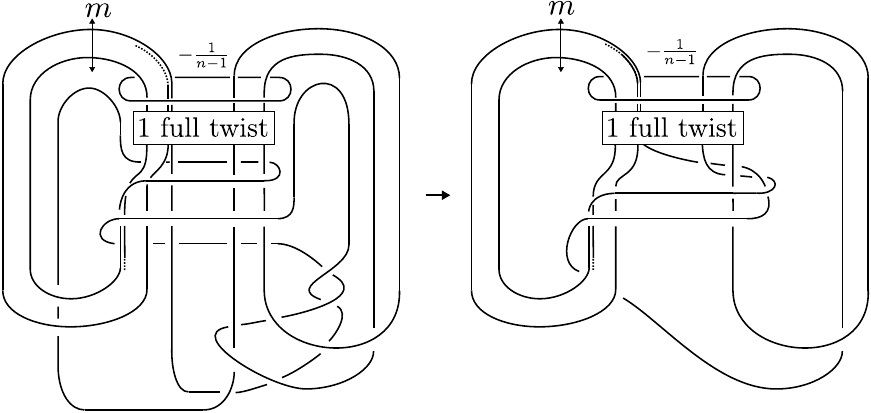}
\caption{The link $K\cup C_2$ (left) and its deformed link (right).}
\label{figure_asy_deform_2}
\end{figure}

To prove the hyperbolicity and asymmetry of $K_{m,n}$, we apply the result of Futer–Purcell–Schleimer \cite{FPS}.
To begin with, we set up the general argument that will be used repeatedly below.

Let $X$ be a cusped hyperbolic $3$-manifold, and suppose that 
\[
\mathrm{sysmin}(L_0)<\mathrm{sys}(X),
\]
where $\mathrm{sys}(X)$ denotes the systole of $X$, namely, the length of its shortest closed geodesic, and $\mathrm{sysmin}(L_0)$ denotes the real number associated to $L_0\ge 10.1$ defined in \cite[Definition 7.25]{FPS}.
Let $\boldsymbol{s}$ be a tuple of slopes on some cusps of $X$ such that the total normalized length $L_X(\boldsymbol{s})$ satisfies
\[
L_X(\boldsymbol{s})\ge L_0.
\]
By \cite[Theorem 7.28]{FPS}, the filled manifold $X(\boldsymbol{s})$ is hyperbolic and the cores of the attached solid tori give the shortest tuple of geodesics in the hyperbolic manifold $X(\boldsymbol{s})$. 
Hence any homeomorphism $\varphi$ on $X(\boldsymbol{s})$ fixes the cores after an isotopy. 
Restricting $\varphi$ to $X$ gives a homeomorphism on $X$.
In particular, if $X$ is asymmetric, then $X(\boldsymbol{s})$ is asymmetric.

\begin{lemma}\label{lem_asy_1}
If $m\ge 20$ and $n\ge 21$, then $K_{m,n}$ is hyperbolic and asymmetric.
\end{lemma}

\begin{proof}
Let $M$ be the cusped manifold obtained by filling the complement of the link $L=K \cup C_0 \cup C_1 \cup C_2 \cup C_3$ along $C_3$ with coefficient $-1$.
We note that $M$ is hyperbolic and asymmetric.
(In fact, $M$ is homeomorphic to the complement of \texttt{L14n61770} in the census.)
Moreover, SnapPy verifies that $\mathrm{sys}(M)>1$.
Set $L_0=10.1$. As shown in \cite[Lemma 7.26]{FPS},
\[
\mathrm{sysmin}(L_0)<\frac{2\pi}{L_0^2-58}=0.142\cdots<\mathrm{sys}(M).
\]
Let $T_0$, $T_1$ and $T_2$ be the cusps corresponding to $C_0$, $C_1$ and $C_2$. 
By the above argument, if the total normalized length $L_M(\boldsymbol{s}) \ge 10.1$, then the filled manifold $M(\boldsymbol{s})=(K\cup C_0\cup C_1 \cup C_2)(*,\frac{1}{m},-\frac{1}{m},-\frac{1}{n-1})$, 
which is the complement of $K_{m,n}$, 
is hyperbolic and asymmetric.
Hence we can conclude that $K_{m,n}$ is asymmetric.

Let $\boldsymbol{s}=(s_0,s_1,s_2)=(\frac{1}{m},-\frac{1}{m},-\frac{1}{n-1})$.
The normalized length of $s_i$ is given as 
\[
L_M(s_i)=\frac{\mathrm{length}(s_i)}{\sqrt{\mathrm{area}(T_i)}},
\]
where $\mathrm{length}(s_i)$ is the length of a geodesic representative of $s_i$ on the cusp $T_i$.
Recall that the total normalized length $L_M(\boldsymbol{s})$ is defined by the formula
\[
\frac{1}{L_M(\boldsymbol{s})^2}=\frac{1}{L_M(s_0)^2}+\frac{1}{L_M(s_1)^2}+\frac{1}{L_M(s_2)^2}.
\]

The cusp shapes of $T_0$, $T_1$ and $T_2$ are given by 
\begin{align*}
z_0&=-0.375710603555430? + 0.594247011057927?*i,\\
z_1&=0.294509307916497? + 0.432091899249841?*i,\\
z_2&=-0.473512401909394? + 0.334485477737059?*i,
\end{align*}
using the standard meridian-longitude basis.
(These are given by verified computations in SnapPy.
The question mark indicates that the preceding digit is possibly wrong by $\pm1$.)
This means that the parallelogram in $\mathbb{C}$ with vertices $0$, $1$, $z_i$ and $z_i+1$ represents the similarity class of the cusp $T_i$ where the meridian corresponds to the edge from $0$ to $1$ and the longitude corresponds to the edge from $0$ to $z_i$.

Thus, with these shapes, the slope $s_0=\frac{1}{m}$ on $T_0$ has length $|mz_0+1|$ and normalized length $L_M(s_0)=\frac{|mz_0+1|}{\sqrt{|\mathrm{Im}(z_0)|}}$.
Similarly, the slope $s_1=-\frac{1}{m}$ on $T_1$ has normalized length $L_M(s_1)=\frac{|mz_1-1|}{\sqrt{|\mathrm{Im}(z_1)|}}$,
and the slope $s_2=-\frac{1}{n-1}$ on $T_2$ has normalized length $L_M(s_2)=\frac{|(n-1)z_2-1|}{\sqrt{\mathrm{Im}(z_2)}}$. 
It follows that the total normalized length $L_M(\boldsymbol{s})$ is monotonically increasing in both $m$ and $n$. 
Therefore, it suffices to consider the case $m=20$ and $n=21$.
By direct computation, we have $L_M(\boldsymbol{s})\ge 10.1$ when $m=20,n=21$.
This proves the claim.
\end{proof}

\begin{lemma}\label{lem_asy_2}
If $3\le m \le19$ and $n\ge 1$,  then $K_{m,n}$ is hyperbolic and asymmetric.
\end{lemma}

\begin{proof}
Let $M$ be the manifold in the proof of Lemma~\ref{lem_asy_1}. 
For $m\in \{3,\dots,19\}$, let $N_m$ be the cusped manifold obtained by filling $M$ along $C_0$ and $C_1$ with coefficients $\frac{1}{m}$ and $-\frac{1}{m}$, respectively.
SnayPy verifies that $N_m$ is hyperbolic and asymmetric.

For each $m\in \{3,\dots,19\}$, we compute $\mathrm{sys}(N_m)$ using SnapPy and choose $L_m\geq 10.1$ such that
\[
\frac{2\pi}{L_m^2-58}<\operatorname{sys}(N_m).
\]
See Table \ref{table:sys}.

\begin{table}[ht]
\caption{$\mathrm{sys}(N_m)$, $L_m$ and $n_m$.}\label{table:sys}
\small
\begin{tabular}{| c | c | c | c || c| c| c| c|}
\hline
$m$  & $\mathrm{sys}(N_m)$     & $L_m$ &  $n_m$ & $m$ & $\mathrm{sys}(N_m)$     & $L_m$ &  $n_m$ \\
\hline
3  & 1.04843826338309   & 10.1 & 10 & 12 & 0.0596468410571235 & 12.8 & 13 \\
4  & 0.712479392039562  &  10.1& 10 & 13 & 0.0503131143290759 & 13.6 & 14  \\
5  & 0.417394449183117  &  10.1& 10 & 14 & 0.0430103201986191 &  14.3&  14\\
6  & 0.273574595743006  &  10.1 &  10 &15 & 0.0371892471952512 &  15.1&  15 \\
7  & 0.193060648170864  &  10.1 & 10  & 16 & 0.0324745469728705 &  15.9& 16\\
8  & 0.143505095158023  &  10.1& 10 &  17 & 0.0286026257272713 &  16.7&  17\\
9  & 0.110849671955907  &   10.8& 11 & 18 & 0.0253839643728564 & 17.5 & 18\\
10 & 0.0881968649176810 &  11.4&  11 & 19 & 0.0226794709765155 &  18.4 &18 \\
11 & 0.0718410577723763 &  12.1& 12 & & &  &  \\
\hline
\end{tabular}
\end{table}

Then, $\mathrm{sysmin}(L_m)<\mathrm{sys}(N_m)$ \cite[Lemma 7.26]{FPS}.
We  find an integer $n_m$ such that the normalized length of the slope $-\frac{1}{n-1}$ on $T_2$ satisfies
\[
L_{N_m}\left(-\frac{1}{n-1}\right)\geq L_m
\]
for all $n\geq n_m$. Then, the filled manifold $N_m(-\frac{1}{n-1})$, which is the complement of $K_{m,n}$, is hyperbolic and asymmetric for all $n\geq n_m$.

The finitely many remaining cases are verified individually by SnapPy.
\end{proof}

\begin{lemma}\label{lem_asy_3}
If $m \ge 20$ and $1\le n\le 20$, then $K_{m,n}$ is hyperbolic and asymmetric.
\end{lemma}

\begin{proof}
Let $M$ be the manifold in the proof of Lemma~\ref{lem_asy_1}. 
For $n\in \{1,\dots,20\}$, let $J_n$ be the cusped manifold obtained by filling $M$ along $C_2$ with coefficient $-\frac{1}{n-1}$.
The rest of the argument is the same as in the proof of Lemma~\ref{lem_asy_2}. 

For each $n\in \{1,\dots,20\}$,  compute $\mathrm{sys}(J_n)$ and choose $L_n\ge 10.1$ such
that 
\[
\frac{2\pi}{L_n^2-58}<\mathrm{sys}(J_n).
\]
As before, we can find an integer $m_n$ such that
the total normalized length of the pair of slopes $\boldsymbol{s}=(\frac{1}{m},-\frac{1}{m})$ on $(T_0,T_1)$
satisfies 
\[
L_{J_n}(\boldsymbol{s})\ge L_n
\]
for all $m \ge m_n$.
See Table \ref{table:sys2}.
Then the rest of the argument goes as before.

\begin{table}[ht]
\caption{$\mathrm{sys}(J_n)$, $L_n$ and $m_n$.}\label{table:sys2}
\small
\begin{tabular}{| c | c | c | c || c| c| c| c|}
\hline
$n$  & $\mathrm{sys}(J_n)$     & $L_n$ &  $m_n$ & $n$ & $\mathrm{sys}(J_n)$     & $L_n$ &  $m_n$ \\
\hline
1  &  1.033321292811034  &10.1  & 20 & 11 & 0.0476049980334768 & 13.8 & 24 \\
2  &  0.909476091800048  & 10.1  & 19 & 12 & 0.0402982084024402 & 14.7 & 26  \\
3  &  0.490006165572169 &  10.1 & 18 & 13 & 0.0345499062187332 & 15.5 &  27\\
4  &  0.303412013129525 &  10.1 &  18 &14 & 0.0299471832695823 &  16.4 &  29 \\
5  &   0.205421443632841 & 10.1  & 18  & 15 & 0.0262051401966472&  17.3 & 30\\
6  &  0.147953649888420 &  10.1 & 18 &  16 & 0.0231221195351999 & 18.2  & 32\\
7  &  0.111495301756357 &  10.7  & 19 & 17 & 0.0205521723582647&  19.1 & 33\\
8 & 0.0869665685961345 & 11.5  &  21 & 18 & 0.0183875811797308 &  20.0  &35 \\
9 & 0.0696958284585769 &  12.2 & 22 & 19&0.0165474544756519 & 21.0 & 36 \\
10 &0.0570866686759814 &  13.0 & 23 & 20 & 0.0149700985970390& 21.9 &  38\\
\hline
\end{tabular}
\end{table}


Remark that the manifold $J_1$ is not asymmetric.
It is homeomorphic to the complement of \texttt{L12n1625}.
However, its only non-trivial symmetry is an involution that exchanges the cusps corresponding to $K$ and $C_1$.
Therefore, for $m$ satisfying the hypothesis of \cite[Theorem 7.28]{FPS}, any homeomorphism of the filled manifold $J_1(\frac{1}{m}, -\frac{1}{m})$ must restrict to a homeomorphism of $J_1$ that is isotopic to the identity. 
Thus, the result for $n=1$ follows in the same manner as for the other cases.
\end{proof}

\begin{proof}[Proof of Theorem \ref{thm:main}]
By Theorem \ref{thm:Kmn}, $K_{m,n}$ is an L–space knot. Lemmas~\ref{lem_asy_1},~\ref{lem_asy_2} and~\ref{lem_asy_3} show that $K_{m,n}$ is hyperbolic and asymmetric if $m \ge 3$ and $n\ge 1$.
\end{proof}



\bibliographystyle{amsplain}

\end{document}